\documentclass[a4paper,10pt]{article}
\usepackage[utf8]{inputenc}

\usepackage{graphicx,amssymb,amsmath,amsfonts,amsthm,amscd,hyperref,textcomp}

\newtheorem{proposition}{Proposition}
\newtheorem{lemma}{Lemma}
\newtheorem{corollary}{Corollary}
\newtheorem{thm}{Theorem}

\newcommand{\CC}{\mathbb C}

\newcommand{\Fp}{{\mathbb F}_p}

\newcommand{\Fpr}{{\mathbb F}_{p^r}}

\newcommand{\Ql}{\bar{\mathbb Q}_\ell}
\newcommand{\Dbc}{D^b_c}

\newcommand{\FF}{\mathcal F}
\newcommand{\GGG}{\mathcal G}
\newcommand{\HH}{\mathcal H}
\newcommand{\LL}{\mathcal L}
\newcommand{\R}{\mathrm R}

\newcommand{\AAA}{\mathbb A}
\newcommand{\GG}{\mathbb G}
\newcommand{\Gm}{{\mathbb G}_m}

\newcommand{\PP}{\mathbb P}
\newcommand{\HHH}{\mathrm H}
\newcommand{\UU}{\mathcal U}
\newcommand{\KK}{\mathcal K}
\newcommand{\Spec}{\mathrm{Spec}}
\newcommand{\ie}{\emph{ie. }}

\title{Explicit local multiplicative convolution of $\ell$-adic sheaves}
\author{Antonio Rojas-Le\'on}

\begin{document}

\maketitle

\renewcommand{\thefootnote}{}
\footnote{Partially supported by MTM2013-46231-P (Ministerio de Econom\'{\i}a y competitividad) and FEDER}

\begin{abstract}
We give explicit formulas for the local multiplicative convolution functors \cite{katz1988gauss,rojas2013local} which express the local monodromies of the convolution of two $\ell$-adic sheaves on the torus $\Gm$ over the algebraic closure of a finite field in terms of the local monodromies of the factors.
\end{abstract}

\section{Introduction}

The effect of global cohomological operations in certain categories on the local properties of the objects on which they operate has been extensively studied. In \cite{laumon1987transformation} Laumon translated the stationary phase principle from functional analysis to Deligne's $\ell$-adic Fourier transform: the local monodromies of the Fourier transform of an $\ell$-adic sheaf on the affine line over a finite field can be determined from those of the original sheaf, via some ``local Fourier transform'' functors. Laumon and Malgrange \cite[2.6.3]{laumon1987transformation} gave conjectural explicit formulas for these functors, which operate on the category of $\ell$-adic representations of the decomposition group of the affine line at a point. These formulas were proved (with some modifications) independently by Fu \cite{fu2010calculation} and Abbes and Saito \cite{abbes2010local}. 

In the category of holonomic $D$-modules on the affine line over the complex numbers, which in many ways behaves like the category of $\ell$-adic sheaves on the affine line over a finite field, there have also been some results in this direction: Bloch and Esnault \cite{bloch2004local} and Garc\'{\i}a L\'opez \cite{lopez2004microlocalization} defined the local Fourier transform functors for D-modules, showing that the local monodromies at the singular points of the Fourier transform of a holonomic $D$-module $M$ are determined by those of $M$. Fang \cite{fang2009calculation} and Sabbah \cite{sabbah2007explicit} gave explicit formulas for these, similar to the ones for the $\ell$-adic case.

In this article we will consider the multiplicative convolution operation on the category of sheaves on the one-dimensional torus over a finite field (note that, since the Fourier transform interchanges additive convolution and tensor product, the formulas for the local Fourier transform immediately give formulas for the local additive convolution, as already noted by Laumon in \cite[2.7]{laumon1987transformation}). In \cite{katz1988gauss}, Katz proves that the convolution of two smooth sheaves on $\Gm$ with tamely ramified monodromy at $0$ and totally wild monodromy at infinity is another sheaf of the same form, and the local monodromies of the convolution can be determined from those of the factors. In \cite{rojas2013local} we extend this to general (perverse) sheaves on $\Gm$: there exist ``local covolution'' functors that give us the monodromies of the convolution of two objects at any point in terms of those of the factors. Here we will give explicit formulas for these functors, similar to the ones 
given in \cite{fu2010calculation} for the local Fourier transform (which are in fact a particular case of these, as we will see in the last section).

Throughout this article, $p$ will be an odd prime, and $k=\overline{{\mathbb F}_p}$ the algebraic closure of the prime field ${\mathbb F}_p$. We will fix another prime $\ell\neq p$, and let ${\mathcal S}(\Gm,\Ql)$ be the category of constructible $\Ql$-sheaves on the one-dimensional torus $\Gm:={\mathbb G}_{m,k}$ and $\Dbc(\Gm,\Ql)$ the corresponding derived category. 

Let $I_0$ and $I_\infty$ denote the inertia groups of $\Gm$ at $0$ and $\infty$ respectively, which are isomorphic to the Galois groups of the henselizations of $k[t]$ at the ideal $(t)$ and of $k[t^{-1}]$ at $(t^{-1})$. We have an exact sequence of groups
$$
1\to P_0\to I_0\to I_0^{tame}\to 1
$$
where $P_0$ is a pro-$p$ group (the wild inertia group) and $I_0^{tame}\cong\prod_{\ell\neq p}{\mathbb Z}_\ell$, and similarly for $I_\infty$. Every sheaf $\FF\in{\mathcal S}(\Gm,\Ql)$ induces continuous $\Ql$-representations of $I_0$ and $I_\infty$, called the local monodromies of $\FF$ at $0$ and $\infty$. We will denote these representations by $\FF_0$ and $\FF_\infty$, or simply by $\FF$ if no confusion can arise.

 To every finite extension $\Fp\subseteq\Fpr$ and multiplicative character $\chi:\Fpr^\times\to\CC^\times$ corresponds a 1-dimensional smooth sheaf $\LL_\chi\in{\mathcal S}({\mathbb G}_{m,\Fpr},\Ql)$ (the Kummer sheaf, \cite[1.4-1.8]{deligne569application}). By extension of scalars to $k$, this gives tamely ramified characters of $I_0$ and $I_\infty$, also denoted by $\LL_\chi$. Every character of $I_0$ or $I_\infty$ of finite order prime to $p$ is isomorphic to one of these. Every tamely ramified (\ie trivial on $P_\infty$) continuous $\Ql$-representation of $I_\infty$ is a direct sum of representations of the form $\KK_{\chi,n}:=\LL_\chi\otimes\UU_n$, where $\chi$ is one such character and $\UU_n$ is the unique (up to isomorphism) unipotent indecomposable representation of dimension $n$ (by taking the Jordan decomposition of the action of a topological generator of $I_\infty^{tame}$). If $\xi:\Fpr^\times\to\CC^\times$ is another character, then $\LL_{\chi\xi}\cong\LL_\chi\otimes\LL_\xi$.

 Let also $\psi:\Fp\to\CC^\times$ be the additive character $\psi(t)=\exp(2\pi it/p)$ and $\LL_\psi\in{\mathcal S}(\AAA^1_k,\Ql)$ the corresponding Artin-Schreier sheaf \cite[1.4-1.8]{deligne569application}. Given a sheaf $\FF\in{\mathcal S}(\Gm,\Ql)$ and a $k$-morphism $h:X\to\Gm$, we will denote by $\FF(h)$ the pull-back sheaf $h^\ast\FF$ on $X$. In particular, given a polynomial $f\in k[t]$ of degree prime to $p$ we can consider the sheaves $\LL_\psi(f)$ and $\LL_\psi(f(1/t))$ on $\Gm$, and their induced representations of $I_\infty$ and $I_0$ respectively. They are characters of slope $\deg(f)$.

 For every positive integer $d$ prime to $p$, the $d$-th power map $[d]:\Gm\to\Gm$ induces injective homomorphisms $I_0\to I_0$ and $I_\infty\to I_\infty$, that can be used to identify $I_\infty$ and $I_0$ with their unique closed subgroups $I_\infty^d$, $I_0^d$ of index $d$. Given a sheaf $\FF\in{\mathcal S}(\Gm,\Ql)$, the pull-back and push-forward of $\FF$ by $[d]$ correspond to restricting the representation $\FF_\infty$ to $I_\infty^d$ and taking the induced representation
of $\FF_\infty$ from $I_\infty^d$ to $I_\infty$. We will denote these representations by $[d]^\ast\FF$ and $[d]_\ast\FF$ respectively.

 Given two objects $K,L\in\Dbc(\Gm,\Ql)$, their \emph{convolution} is defined to be the object
$$
K\ast L:=\R\mu_!(K\boxtimes L)\in\Dbc(\Gm,\Ql)
$$
where $\mu:\Gm\times\Gm\to\Gm$ is the multiplication map. If $K=\FF[1]$ and $L=\GGG[1]$, where $\FF,\GGG\in {\mathcal S}(\Gm,\Ql)$ are smooth sheaves which are tamely ramified at zero and totally wild (\ie with no $P_\infty$-fixed elements) at infinity, then $K\ast L={\mathcal H}[1]$ is another object of the same form \cite[Theorem 5.1]{katz1988gauss}. Moreover, the local monodromy of $\mathcal H$ at infinity depends only of those of $\FF$ and $\GGG$: there exists a bi-exact functor $LC_{(\infty,\infty)}^{\infty}:{\mathcal R}^w_\infty\times{\mathcal R}^w_\infty\to{\mathcal R}^w_\infty$ (where ${\mathcal R}^w_\infty$ is the category of totally wild continuous $\Ql$-representations of $I_\infty$) such that the monodromy of $\mathcal H$ at infinity is given by $LC_{(\infty,\infty)}^{\infty}(\FF_\infty,\GGG_\infty)$ \cite[chapter 6]{katz1988gauss}.

More generally, if $K=\FF[1]$, $L=\GGG[1]$ are irreducible perverse objects (where $\FF,\GGG\in {\mathcal S}(\Gm,\Ql)$ are irreducible middle extension sheaves) \cite[8.1]{katz1990esa} and $K\ast L={\mathcal H}[1]$ with ${\mathcal H}\in {\mathcal S}(\Gm,\Ql)$, then there are bi-exact functors $LC_{(a,b)}^{c}:{\mathcal R}_a\times{\mathcal R}_b\to{\mathcal R}_c$ for every $(a,b,c)\in(\PP^1_k)^3$ in the closure $\overline Z$ of $Z:=\{(x,y,xy)|x,y\in k\}$ such that the local monodromy (the wild part if $c=0$ or $\infty$) of $\mathcal H$ at $c$ is the direct sum of $LC_{(a,b)}^{c}(\FF_a,\GGG_b)$ for every $(a,b)$ such that $(a,b,c)\in \overline Z$ \cite[Theorems 9,17]{rojas2013local}.

In this article we give explicit formulas for these local convolution functors for a wide class of representations (which include those that arise from arithmetic or geometric applications). Namely, we consider representations of $I_\infty$ if the form $[a]_\ast(\LL_{\psi}(f)\otimes\KK_{\chi,n})$, where $a$ is a prime to $p$ integer, $f\in k[t]$ is a polynomial of degree $d$ prime to $p$ and $\chi:\Fpr^\times\to\CC^\times$ is a multiplicative character for some $r\geq 1$. Even though not every continuous $\Ql$-representation of $I_\infty$ is of this form, most interesting ones are. See \cite[Proposition 0.5]{fu2010calculation} for a discussion on this topic.

From the construction of the different $LC_{(a,b)}^{c}$ functors in \cite{rojas2013local}, we see that all of them can be defined from $LC_{(\infty,\infty)}^{\infty}$ and $LC_{(0,\infty)}^{\infty}$ (from these two one can construct the local Fourier transform functors as special cases \cite[Proposition 8.1.12]{katz1990esa}, and then define all other local convolution functors by recursive composition of these two with the inversion and the local Fourier transform functors). So we will focus on these two, and explain in the last section how to derive formulas for the remaining ones from these.

Let $\FF=[a]_\ast(\LL_{\psi}(f)\otimes\KK_{\chi,n})$ and $\GGG=[b]_\ast(\LL_\psi(g)\otimes\KK_{\xi,m})$ where $a,b,n,m$ are positive integers, $a,b$ are prime to $p$, $f,g\in k[x]$ are polynomials of degrees $d,e$ prime to $p$, and $\chi,\xi$ are multiplicative characters of some finite extension of $\Fp$. Our main results provide explicit formulas for the representations $LC_{(\infty,\infty)}^{\infty}(\FF,\GGG)$ and $LC_{(0,\infty)}^{\infty}(\iota^\ast\FF,\GGG)$ (where $\iota:\Gm\to\Gm$ is the inversion map). We will assume that $a$ and $b$ are relatively prime, since $[r]_\ast(K\ast L)\cong([r]_\ast K)\ast([r]_\ast L)$ for every $r\geq 1$ and $K,L\in\Dbc(\Gm,\Ql)$ \cite[Theorem 5.1(10)]{katz1988gauss}. Let $c$ be the g.c.d. of $d$ and $e$, and write $d=cd'$, $e=ce'$ so that $d'$ and $e'$ are relatively prime.

We construct the following Laurent polynomial $H(z,t)\in k[[t^{-1}]][z,z^{-1}]$:
$$
H(z,t)=t^{-de/c}\left(f(t^{e'}z^b)+g(t^{d'}z^{-a})\right)
$$
Its reduction modulo $t^{-1}$ is given by
$$
\tilde H(z)=f_dz^{bd}+g_ez^{-ae}=z^{-ae}(f_dz^{bd+ae}+g_e)\in k[z,z^{-1}]
$$
where $f_d$ and $g_e$ are the leading coefficients of $f$ and $g$. Assume that $bd+ae$ is prime to $p$. Then the derivative of $H(z)$
$$
\tilde H'(z)=z^{-ae-1}(bdf_dz^{bd+ae}-aeg_e)
$$
has $bd+ae$ simple roots in $k^\times$, $\alpha_0,\ldots,\alpha_{bd+ae-1}$ where $\alpha_i=\alpha_0\zeta^i$, $\zeta\in k$ being a primitive $(bd+ae)$-th root of unity. By Hensel's lemma, each of them can be lifted to a root $z_i(t^{-1})$ of $\frac{\partial}{\partial z}H(z,t)$ in $k[t^{-1}]^h_{(t^{-1})}$ (the henselization of the localization of $k[t^{-1}]$ at the ideal $(t^{-1})$) such that $z_i(t^{-1})\equiv\alpha_i$ modulo $t^{-1}$. Let
$$
h_i(t)=t^{de/c}H(z_i(t^{-1}),t)\in t^{de/c}\cdot k[[t^{-1}]]
$$
for $0\leq i\leq bd+ae-1$.

\begin{thm}\label{t1}
 Suppose that $a,b,d,e$ and $bd+ae$ are prime to $p$. Then
$$
LC_{(\infty,\infty)}^{\infty}(\FF,\GGG)\cong [bd'+ae']_\ast\left(\bigoplus_{i=0}^{c-1}\LL_\psi(h_i)\otimes\LL_{\rho^{de}}\otimes\LL_{\chi^{e'}\xi^{d'}}\otimes\UU_n\otimes\UU_m\right)
$$
where $\rho$ is the order two character of $I_\infty$.
\end{thm}

The local convolution $LC_{(0,\infty)}^{\infty}(\iota^\ast\FF,\GGG)$ is zero if the slope of $\iota^\ast\FF$ at zero $(d/a)$ is less than or equal to the slope of $\GGG$ at infinity $(e/b)$, by \cite[Proposition 13]{rojas2013local}. So let us assume that $bd>ae$. We construct the Laurent polynomial in $k[[t^{-1}]][z,z^{-1}]$:
$$
H(z,t)=t^{-de/c}\left(f(t^{e'}z^{-b})+g(t^{d'}z^{-a})\right)
$$
whose reduction modulo $t^{-1}$ is given by
$$
\tilde H(z)=f_dz^{-bd}+g_ez^{-ae}=z^{-bd}(f_d+g_ez^{bd-ae}).
$$ 
If $bd-ae$ is prime to $p$, the derivative of $\tilde H(z)$
$$
\tilde H'(z)=-z^{-bd-1}(bdf_d+aeg_ez^{bd-ae})
$$
has $bd-ae$ simple roots in $k^\times$, $\alpha_0,\ldots,\alpha_{bd-ae-1}$ where $\alpha_i=\alpha_0\zeta^i$, $\zeta\in\bar k$ being a primitive $(bd-ae)$-th root of unity. By Hensel's lemma, each of them can be lifted to a root $z_i(t^{-1})$ of $\frac{\partial}{\partial z}H(z,t)$ in $k[t^{-1}]^h_{(t^{-1})}$ such that $z_i(t^{-1})\equiv\alpha_i$ modulo $t^{-1}$. Let
$$
h_i(t)=t^{de/c}H(z_i(t^{-1}),t)\in t^{de/c}\cdot k[[t^{-1}]]
$$
for $0\leq i\leq bd-ae-1$.
 
\begin{thm}\label{t2}
 Suppose that $a,b,d,e$ and $bd-ae$ are prime to $p$. Then
$$
LC_{(0,\infty)}^{\infty}(\iota^\ast\FF,\GGG)\cong [bd'-ae']_\ast\left(\bigoplus_{i=0}^{c-1}\LL_\psi(h_i)\otimes\LL_{\rho^{de}}\otimes\LL_{\chi^{e'}\xi^{d'}}\otimes\UU_n\otimes\UU_m\right)
$$
where $\rho$ is the order two character of $I_\infty$.
\end{thm}

We will prove Theorems \ref{t1} and \ref{t2} in sections 2 and 3 respectively. In section 4 we will see how one can deduce formulas for the local Fourier transforms and the other local colvolution functors from these two.

\section{The local convolution $LC_{(\infty,\infty)}^{\infty}$}

In this section we will prove Theorem \ref{t1} about the local convolution $LC_{(\infty,\infty)}^{\infty}(\FF,\GGG)$. We will first compute its restriction to the subgroup of index $bd'+ae'$ of $I_\infty$.

\begin{proposition}\label{p1} With the notation defined in the previous section, suppose that $bd+ae$ is prime to $p$. Then
 $$
[bd'+ae']^\ast LC_{(\infty,\infty)}^{\infty}(\FF,\GGG)\cong\bigoplus_{i=0}^{bd+ae-1}(\LL_{\psi(h_i)}\otimes\LL_{\rho^{de/c}}\otimes\LL_{\chi^{e'}\xi^{d'}}\otimes \UU_n\otimes \UU_m)
$$
\end{proposition}

We will view $\FF$ and $\GGG$ indistinctly as smooth sheaves on $\Gm$ or as representations of $I_\infty$. If $(\FF[1])\ast(\GGG[1])\cong{\mathcal H}[1]$, we will write by abuse of language ${\mathcal H}=\FF\ast\GGG$. Let $\pi_1,\pi_2,\mu:\Gm\times\Gm\to\Gm$ denote the projections and the multiplication map. Then
$$\FF\ast\GGG=\R^1\mu_!([a]_\ast(\LL_{\psi(f)}\otimes\KK_{\chi,n})\boxtimes[b]_\ast(\LL_{\psi(g)}\otimes\KK_{\xi,m}))=
$$
$$
=\R^1\mu_![a,b]_\ast((\LL_{\psi(f)}\otimes\KK_{\chi,n} )\boxtimes(\LL_{\psi(g)}\otimes\KK_{\xi,m}))=
$$
$$
=\R^1\sigma_!((\LL_{\psi(f)}\otimes\KK_{\chi,n})\boxtimes(\LL_{\psi(g)}\otimes\KK_{\xi,m}))
$$
where $[a,b]:\Gm\times\Gm\to\Gm\times\Gm$ is the finite \'etale map given by $(x,y)\mapsto(x^a,y^b)$ and $\sigma(x,y)=\mu\circ[a,b](x,y)=x^ay^b$.

Let $\alpha,\beta\in{\mathbb Z}$ be integers such that $\alpha a+\beta b=1$. Let $\phi:\Gm\times\Gm\to\Gm\times\Gm$ be the morphism given by $\phi(w,t)=(w^bt^\alpha,w^{-a}t^\beta)$. Then $\phi$ is an automorphism with inverse $\phi^{-1}(x,y)=(x^\beta y^{-\alpha},x^ay^b)$. In particular, $\sigma=\pi_2\phi^{-1}$, so
$$
\FF\ast\GGG=\R^1\sigma_!((\LL_{\psi(f)}\otimes\KK_{\chi,n})\boxtimes(\LL_{\psi(g)}\otimes\KK_{\xi,m}))=
$$
$$
=\R^1\pi_{2!}\phi^\ast((\LL_{\psi(f)}\otimes\KK_{\chi,n})\boxtimes(\LL_{\psi(g)}\otimes\KK_{\xi,m}))
$$

If we denote by $x,t$ the variables in the first and second factor of $\Gm\times\Gm$, we can write
$$\FF\ast\GGG\cong\R^1\pi_{2!}\phi^\ast(\LL_{\psi(f)}(x)\otimes\KK_{\chi,n}(x)\otimes\LL_{\psi(g)}(t)\otimes\KK_{\xi,m}(t))=
$$
$$
=\R^1\pi_{2!}(\LL_{\psi(f)}(w^bt^\alpha)\otimes\KK_{\chi,n}(w^bt^\alpha)\otimes\LL_{\psi(g)}(w^{-a}t^\beta)\otimes\KK_{\xi,m}(w^{-a}t^\beta))
$$

By proper base change, we get
$$
[bd'+ae']^\ast(\FF\ast\GGG)\cong
$$
$$
\cong\R^1\pi_{2!}(Id,[bd'+ae'])^\ast(\LL_{\psi(f)}(w^bt^\alpha)\otimes\KK_{\chi,n}(w^bt^\alpha)\otimes\LL_{\psi(g)}(w^{-a}t^\beta)\otimes\KK_{\xi,m}(w^{-a}t^\beta))=
$$
$$
=\R^1\pi_{2!}(\LL_{\psi(f)}(w^bt^{\alpha(bd'+ae')})\otimes\KK_{\chi,n}(w^bt^{\alpha(bd'+ae')})\otimes
$$
$$
\otimes\LL_{\psi(g)}(w^{-a}t^{\beta(bd'+ae')})\otimes\KK_{\xi,m}(w^{-a}t^{\beta(bd'+ae')}))=
$$
$$
=\R^1\pi_{2!}(\LL_{\psi(f)}((wt^{\alpha d'-\beta e'})^bt^{e'})\otimes\KK_{\chi,n}((wt^{\alpha d'-\beta e'})^bt^{e'})\otimes
$$
$$
\otimes\LL_{\psi(g)}((wt^{\alpha d'-\beta e'})^{-a}t^{d'})\otimes\KK_{\xi,m}((wt^{\alpha d'-\beta e'})^{-a}t^{d'})
$$

Since we are interested in the monodromy at infinity, let us specialize at that point. Let $S=(\PP^1)^h_\infty=\mathrm{Spec}\; k[t^{-1}]^h_{(t^{-1})}$ be the henselization of $\PP^1$ at $\infty$, $\eta=\mathrm{Spec}\;(\mathrm{Frac}\; k[t^{-1}]^h_{(t^{-1})})\hookrightarrow S$ its generic point, and $\bar\eta$ a geometric point over $\eta$. For $A=S$ or $\eta$, let $\pi_1:\GG_{m,A}:=\Gm\times_k A\to \Gm$ and $\pi_2:\GG_{m,A}\to A$ be the projection maps. If $j:\Gm\to\PP^1$ is the inclusion, given a sheaf ${\mathcal H}$ on $\Gm$ we obtain sheaves on $\eta$ and $S$ by restriction and extension by zero respectively. By abuse of language, we will also denote these sheaves by ${\mathcal H}$. We then have, as $I_\infty$-representations (\ie as sheaves on $\eta$),
$$
[bd'+ae']^\ast(\FF\ast\GGG)\cong\R^1\pi_{2!}(\LL_{\psi(f)}((wt^{\alpha d'-\beta e'})^bt^{e'})\otimes\KK_{\chi,n}((wt^{\alpha d'-ta e'})^bt^{e'})\otimes
$$
$$
\otimes\LL_{\psi(g)}((wt^{\alpha d'-\beta e'})^{-a}t^{d'})\otimes\KK_{\xi,m}((wt^{\alpha d'-\beta e'})^{-a}t^{d'})
$$
Let us now consider the $\eta$-automorphism $\varphi:\GG_{m,\eta}\to\GG_{m,\eta}$ given by $w\mapsto wt^{\alpha d'-\beta e'}$, whose inverse is given by $z\mapsto zt^{\beta e'-\alpha d'}$. Since $\pi_2\varphi=\pi_2$, we have
\begin{equation}\label{eq1}
[bd'+ae']^\ast(\FF\ast\GGG)\cong\R^1\pi_{2!}\varphi_\ast(\LL_{\psi(f)}((wt^{\alpha d'-be'})^b t^{e'})\otimes\KK_{\chi,n}((wt^{\alpha d'-\beta e'})^b t^{e'})\otimes
\end{equation}
$$
\otimes\LL_{\psi(g)}((wt^{\alpha d'-\beta e'})^{-a}t^{d'})\otimes\KK_{\xi,m}((wt^{\alpha d'-\beta e'})^{-a}t^{d'})=
$$
$$
=\R^1\pi_{2!}\left(\LL_{\psi(f)}(z^bt^{e'})\otimes\KK_{\chi,n}(z^bt^{e'})
\otimes\LL_{\psi(g)}(z^{-a}t^{d'})\otimes\KK_{\xi,m}(z^{-a}t^{d'})\right)=
$$
$$
=\R^1\pi_{2!}\left(\LL_{\psi}(f(z^bt^{e'})+g(z^{-a}t^{d'}))\otimes\KK_{\chi,n}(z^bt^{e'})\otimes\KK_{\xi,m}(z^{-a}t^{d'})\right)\cong
$$
$$
\cong\R^1\pi_{2!}\left(\LL_{\psi}(f(z^bt^{e'})+g(z^{-a}t^{d'}))\otimes\LL_{\chi^{e'}\xi^{d'}}(t)\otimes\LL_{\chi^b\xi^{-a}}(z)\otimes\UU_{n}(z^bt^{e'})\otimes\UU_{m}(z^{-a}t^{d'})\right)\cong
$$
$$
\cong\LL_{\chi^{e'}\xi^{d'}}\otimes\R^1\pi_{2!}\left(\LL_{\psi}(f(z^bt^{e'})+g(z^{-a}t^{d'}))\otimes\LL_{\chi^b\xi^{-a}}(z)\otimes\UU_{n}(z^bt^{e'})\otimes\UU_{m}(z^{-a}t^{d'})\right)
$$
by the projection formula.

Let $\HH$ be the sheaf $\LL_{\psi}(f(z^bt^{e'})+g(z^{-a}t^{d'}))\otimes\LL_{\chi^b\xi^{-a}}(z)\otimes\UU_{n}(z^bt^{e'})\otimes\UU_{m}(z^{-a}t^{d'})$ on $\GG_{m,\eta}$, extended by zero to $\PP^1_S:=\PP^1_k\times_k S$. We will study $\R^1\pi_{2!}\HH$ via the vanishing cycles complex $\R\Phi(\HH)$ of the sheaf $\HH$ on $\PP^1_S$ with respect to the projection $\bar\pi_2:\PP^1_S\to S$ (which is the same as the nearby cycles object, since the sheaf vanishes on $\PP^1_\infty\cong\PP^1\times_k\infty$) (cf. \cite{deligne1973groupes} for its definition and properties). We have, by \cite[XIII.2.1.8]{deligne1973groupes},
\begin{equation}\label{eq2}
\R\pi_{2!}\left(\LL_{\psi}(f(z^bt^{e'})+g(z^{-a}t^{d'}))\otimes\LL_{\chi^b\xi^{-a}}(z)\otimes\UU_{n}(z^bt^{e'})\otimes\UU_{m}(z^{-a}t^{d'})\right)\cong
\end{equation}
$$
\cong\R\bar\pi_{2\ast}\HH=\R\Gamma(\PP^1_\infty,\R\Phi(\HH))
$$

The following result is the core of the proof of Theorem \ref{t1}.

\begin{lemma}\label{l1}
 $\R^i\Phi(\HH)=0$ for $i\neq 1$. The sheaf $\R^1\Phi(\HH)$ is supported on $\overline Z\times\{\infty\}$, where $\overline Z$ is the set of $(bd+ae)$-th roots of $aeg_e/bdf_d$. For every such root $\alpha_i$, let $z_i(t^{-1})\in k[t^{-1}]^h_{(t^{-1})}$ be the only root of $\frac{\partial}{\partial z}H(z,t)$ such that $z_i(t^{-1})\equiv\alpha_i$ mod $t^{-1}$, and $h_i(t)=t^{de/c}H(z_i(t^{-1}),t)\in t^{de/c} k[[t^{-1}]]$. Then $\R^1\Phi(\HH)_{(\alpha_i,\infty)}$ has dimension $mn$, and $I_\infty$ acts on it via the representation $\LL_{\psi}(h_i)\otimes\LL_{\rho^{de/c}}\otimes\UU_n\otimes\UU_m$.
\end{lemma}

Combined with equations \ref{eq1} and \ref{eq2}, this lemma proves Proposition \ref{p1}.

\begin{proof}

We have
$$
\R\Phi(\HH)=\R\Phi\left(\LL_{\psi}(t^{de/c}H(z,t))\otimes\LL_{\chi^b\xi^{-a}}(z)\otimes\UU_{n}(z^bt^{e'})\otimes\UU_{m}(z^{-a}t^{d'})\right)
$$
where the involved sheaves are extended by zero to $\PP^1\times S$ as needed.

Let $\theta:\GG_{m,S} \to\AAA^1_S$ be the finite $S$-morphism given by $H(z,t)$, which extends uniquely to a finite $S$-morphism $\theta:\PP^1_S\to\PP^1_S$. Since the vanishing cycles functor commutes with push-forward by proper maps \cite[XIII.2.1.7]{deligne1973groupes} we have
\begin{equation}\label{theta}
\theta_\ast\R\Phi(\HH)=\theta_\ast\R\Phi\left(\LL_{\psi}(t^{de/c}H(z,t))\otimes\LL_{\chi^b\xi^{-a}}(z)\otimes\UU_{n}(z^bt^{e'})\otimes\UU_{m}(z^{-a}t^{d'})\right)\cong
\end{equation}
$$
\cong\R\Phi\left(\theta_\ast(\LL_{\psi}(t^{de/c}H(z,t))\otimes\LL_{\chi^b\xi^{-a}}(z)\otimes\UU_{n}(z^bt^{e'})\otimes\UU_{m}(z^{-a}t^{d'}))\right)\cong
$$
$$
\cong
\R\Phi\left(\LL_{\psi}(t^{de/c}u)\otimes\theta_\ast(\LL_{\chi^b\xi^{-a}}(z)\otimes\UU_{n}(z^bt^{e'})\otimes\UU_{m}(z^{-a}t^{d'}))\right)
$$
where the last isomorphism comes from the projection formula and we denote by $u$ the coordinate in the codomain of $\theta$.

Let ${\mathcal J}:=\theta_\ast(\LL_{\chi^b\xi^{-a}}(z)\otimes\UU_{n}(z^bt^{e'})\otimes\UU_{m}(z^{-a}t^{d'}))$, and let $W=\theta(Z)\subset\AAA^1_S$, where $Z=\{z_i(t^{-1})|i=0,\ldots,bd+ae-1\}$. We claim that the object $\R\Phi(\LL_{\psi}(t^{de/c}u)\otimes{\mathcal J})$ is supported on $\overline W\cup\{\infty\}$, where $\overline W=\{w\mbox{ mod }t^{-1}|w\in W\}$ is the specialization of $W$. Since $\mathcal J$ is a succesive extension of copies of $\theta_\ast(\LL_{\chi^b\xi^{-a}}(z))$, it suffices to show that  $\R\Phi(\LL_{\psi}(t^{de/c}u)\otimes\theta_\ast(\LL_{\chi^b\xi^{-a}}(z)))$ is supported on $\overline W\cup\{\infty\}$. Note that $\theta:\GG_{m,S}\to\AAA^1_S$ is finite \'etale of degree $bd+ae$ over the complement of $W$. In particular, $\theta_\ast(\LL_{\chi^b\xi^{-a}}(z))$ is a smooth sheaf on $\PP^1_S\backslash(W\cup\{\infty\})$. The fact that $\R\Phi(\LL_{\psi}(t^{de/c}u)\otimes\theta_\ast(\LL_{\chi^b\xi^{-a}}(z)))$ is supported on $\overline W\cup\{\infty\}$ is then a consequence of the fact 
that the sheaf $\LL_{\psi}(t^{de/c}u)$ is universally strongly locally acyclic with respect to $\bar\pi_2$ \cite[1.3.1.2,1.3.1.3]{laumon1987transformation}, being obtained by base change from $\LL_{\psi}(tu)$.

Therefore
\begin{equation}\label{decomposition}
\R\bar\pi_{2*}\left(\LL_\psi(t^{de/c}u)\otimes{\mathcal J}\right)\cong
\bigoplus_{w\in\overline W\cup\infty}\R\Phi\left(\LL_\psi(t^{de/c}u)\otimes{\mathcal J}\right)_{(w,\infty)}
\end{equation}

From (\ref{theta}) we deduce that $\R\Phi(\HH)$ is punctual, and in fact supported on a subset of $\theta^{-1}(\theta(\overline Z)\cup\{\infty\})$. Since we know a priori that $\HHH^i(\PP^1_\infty,\R\Phi(\HH))=\R^i\bar\pi_{2\ast}{\mathcal H}=0$ for $i\neq 1$ \cite[Theorem 5.1.1]{katz1988gauss}, this implies in particular that $\R^i\Phi(\HH)=0$ for $i\neq 1$.

Let $\alpha_i\in\overline Z$, $z_i(t^{-1})\in Z$ such that $z_i\equiv\alpha_i$ mod $t^{-1}$, and $w_i=\tilde H(z_i)\in\overline W$, and consider the restriction of the sheaf $\mathcal J$ to the henselization $(\AAA^1_S)^h_{(w_i,\infty)}$. Denote by $\theta_i:(\GG_{m,S})^h_{(\alpha_i,\infty)}\to(\AAA^1_S)^h_{(w_i,\infty)}$ the restriction of $\theta$. Since $\theta$ is a finite map, ${\mathcal J}_i:=\theta_{i\ast}(\LL_{\chi^b\xi^{-a}}(z)\otimes\UU_{n}(z^bt^{e'})\otimes\UU_{m}(z^{-a}t^{d'}))$ is a direct summand of $\mathcal J$. Moreover, given that $\LL_{\chi^b\xi^{-a}}(z)$ is a smooth sheaf on $\GG_{m,S}$, it is trivial on $(\GG_{m,S})^h_{(\alpha_i,\infty)}$, so ${\mathcal J}_i\cong\theta_{i\ast}(\UU_{n}(z^bt^{e'})\otimes\UU_{m}(z^{-a}t^{d'}))$.

Denote by $K$ the field $k[t^{-1}]^h_{(t^{-1})}$, fraction field of the henselization of $k[t^{-1}]$ at the ideal $(t^{-1})$, so that $\eta=\Spec\;K$. The closed immersion $\eta\hookrightarrow (\GG_{m,S})^h_{(\alpha_i,\infty)}\backslash\varpi^{-1}(\infty)$ (where $\varpi:(\GG_{m,S})^h_{(\alpha_i,\infty)}\to S$ is the projection) associated to the residue map $K[z]_{(z-\alpha_i)}^h\to K$ induces an isomorphism $\pi_1(\eta)\cong\pi_1((\GG_{m,S})^h_{(\alpha_i,\infty)}\backslash\varpi^{-1}(\infty))$ \cite[Proposition I.4.4]{milne1980ecv}. In particular, two smooth sheaves on $(\GG_{m,S})^h_{(\alpha_i,\infty)}\backslash\varpi^{-1}(\infty)$ are isomorphic if and only if their restrictions to $\eta$ are. We apply this fact to $\UU_{n}(z^bt^{e'})$: as a representation of $\pi_1(\eta)$, it is $\UU_{n}(\alpha_i^bt^{e'})$, which is indecomposable and unipotent (being the restriction of an indecomposable unipotent representation to a subgroup of finite index $e'$). So it must be isomorphic to $\UU_n(t)$, which is the 
unique such representation up to isomorphism. Therefore $\UU_{n}(z^bt^{e'})$ is isomorphic to $\UU_n(t)$ on $(\GG_{m,S})^h_{(\alpha_i,\infty)}\backslash\varpi^{-1}(\infty)$, and so is $\UU_m(z^{-a}t^{d'})$ to $\UU_m(t)$. We deduce that $\theta_{i\ast}(\UU_{n}(z^bt^{e'})\otimes\UU_{m}(z^{-a}t^{d'}))\cong\theta_{i\ast}(\UU_{n}(t)\otimes\UU_m(t))\cong\theta_{i\ast}(\theta_i^\ast\UU_n(t)\otimes\theta_i^ \ast\UU_m(t))\cong\theta_{i\ast}(\theta_i^\ast(\UU_n(t)\otimes\UU_m(t)))\cong\UU_n(t)\otimes\UU_m(t)\otimes\theta_{i\ast}\Ql$.

Let $\beta_i(z,t)=H(z_i(t^{-1})\cdot z,t)-\tilde h_i(t)$, where $\tilde h_i(t)=H(z_i(t^{-1}),t)$. Then $\beta_i(z,t)$ satisfies the hypotheses of \cite[Lemma 1.3]{fu2010calculation} (centered at $(1,\infty)$ instead of $(1,0)$), namely: $\beta_i(1,t)=0$, $\frac{\partial\beta}{\partial z}(1,t)=0$ and $\frac{\partial^2\beta_i}{\partial z^2}(1,\infty)\neq 0$. Therefore, by \emph{loc.cit.} there is an $S$-isomorphism $\tilde\phi:(\GG_{m,S})^h_{(1,\infty)}\to(\AAA^1_S)^h_{(0,\infty)}\cong(\AAA^2_k)^h_{(0,\infty)}$ such that $\beta_i=\omega\circ\tilde\phi:(\GG_{m,S})^h_{(1,\infty)}\to(\AAA^1_S)^h_{(0,\infty)}$, where the $S$-morphism $\omega:(\AAA^1_S)^h_{(0,\infty)}\to(\AAA^1_S)^h_{(0,\infty)}$ is given by $z\mapsto z^2$.

By composing $\tilde\phi$ with the $S$-isomorphism $(\GG_{m,S})^h_{(\alpha_i,\infty)}\to(\GG_{m,S})^h_{(1,\infty)}$ defined by $z\mapsto z_i(t^{-1})^{-1}z$ we obtain an $S$-isomorphism $\phi:(\GG_{m,S})^h_{(\alpha_i,\infty)}\to(\AAA^1_S)^h_{(0,\infty)}$ such that $\omega\circ\phi(z,t)=\beta_i(z_i(t^{-1})^{-1}z,t)=\tilde\theta(z,t):=H(z,t)-\tilde h_i(t)$. Let us also denote by $\delta_i:(\AAA^1_S)^h_{(0,\infty)}\to(\AAA^1_S)^h_{(w_i,\infty)}$ the translation defined by $z\mapsto z+\tilde h_i(t)$, so that $\theta_i=\delta_i\tilde\theta_i$. Then 
$$\theta_{i\ast}\Ql=\delta_{i\ast}\tilde\theta_{i\ast}\Ql=\delta_{i\ast}\omega_\ast\phi_\ast\Ql=
$$
$$
=\delta_{i\ast}\omega_\ast\Ql\cong\delta_{i\ast}(\Ql\oplus\LL_\rho(u))=\Ql\oplus\LL_\rho(u-\tilde h_i(t))
$$
where $\rho$ is the unique character of order $2$ of $I_\infty$, and in particular we have injections
$$
\LL_\rho(u-\tilde h_i(t))\otimes\UU_n(t)\otimes\UU_m(t)\hookrightarrow\theta_{i\ast}(\LL_{\chi^b\xi^{-a}}(z)\otimes\UU_{n}(z^bt^{e'})\otimes\UU_{m}(z^{-a}t^{d'}))
$$
and
$$
\R^1\Phi(\LL_\psi(t^{de/c}u)\otimes\LL_\rho(u-\tilde h_i(t))\otimes\UU_n(t)\otimes\UU_m(t))_{(w_i,\infty)}\hookrightarrow
$$
$$
\hookrightarrow\R^1\Phi(\LL_\psi(t^{de/c}u)\otimes\theta_{i\ast}(\LL_{\chi^b\xi^{-a}}(z)\otimes\UU_{n}(z^bt^{e'})\otimes\UU_{m}(z^{-a}t^{d'})))_{(w_i,\infty)}
$$

By the projection formula, we have
$$
\R^1\Phi(\LL_\psi(t^{de/c}u)\otimes\LL_\rho(u-\tilde h_i(t))\otimes\UU_n(t)\otimes\UU_m(t))_{(w_i,\infty)}\cong
$$
$$
\cong
\UU_n(t)\otimes\UU_m(t)\otimes\R^1\Phi(\LL_\psi(t^{de/c}u)\otimes\LL_\rho(u-\tilde h_i(t)))_{(w_i,\infty)}
$$

And using the $S$-isomorphism $\delta_i:(\AAA^1_S)^h_{(0,\infty)}\to(\AAA^1_S)^h_{(w_i,\infty)}$, we get
$$
\R^1\Phi(\LL_\psi(t^{de/c}u)\otimes\LL_\rho(u-\tilde h_i(t))_{(w_i,\infty)}\cong
$$
$$
\cong\R^1\Phi(\LL_\psi(t^{de/c}(u+\tilde h_i(t)))\otimes\LL_\rho(u))_{(0,\infty)}\cong
$$
$$
\cong\LL_\psi(t^{de/c}\tilde h_i(t))\otimes\R^1\Phi(\LL_\psi(t^{de/c}u)\otimes\LL_\rho(u))_{(0,\infty)}\cong
$$
$$
\cong\LL_\psi(t^{de/c}\tilde h_i(t))\otimes[de/c]^\ast\R^1\Phi(\LL_\psi(tu)\otimes\LL_\rho(u))_{(0,\infty)}
$$
by the vanishing cycles base change theorem \cite[Proposition 3.7]{deligne569finitude}. But $\R^1\Phi(\LL_\psi(tu)\otimes\LL_\rho(u))_{(0,\infty)}$ is Laumon's local Fourier transform functor ${\mathcal F}_\psi^{(0,\infty)}$ \cite[2.4.2.3]{laumon1987transformation} applied to $\LL_\rho$, which is isomorphic to $\LL_\rho$ itself \cite[2.5.3.1]{laumon1987transformation}. So
$$
\LL_\psi(t^{de/c}\tilde h_i)\otimes\LL_{\rho^{de/c}}\otimes\UU_n\otimes\UU_m\hookrightarrow
\left(\R^1\Phi(\LL_\psi(t^{de/c}u)\otimes{\mathcal J}_i)\right)_{(w_i,\infty)}
$$

Using again that the vanishing cycle functor commutes with push-forwards by finite maps (this time applied to $\theta_i$), together with the projection formula, we get
$$
\left(\R^1\Phi(\LL_\psi(t^{de/c}u)\otimes{\mathcal J}_i)\right)_{(w_i,\infty)}\cong
$$
$$
\cong\theta_{i\ast}\left[\R^1\Phi\left(\LL_{\psi}(t^{de/c}H(z,t))\otimes\LL_{\chi^b\xi^{-a}}(z)\otimes\UU_{n}(z^bt^{e'})\otimes\UU_{m}(z^{-a}t^{d'})\right)_{(\alpha_i,\infty)}\right]\cong
$$
$$
\cong\R^1\Phi\left(\LL_{\psi}(t^{de/c}H(z,t))\otimes\LL_{\chi^b\xi^{-a}}(z)\otimes\UU_{n}(z^bt^{e'})\otimes\UU_{m}(z^{-a}t^{d'})\right)_{(\alpha_i,\infty)}
$$
so
$$
\LL_\psi(t^{de/c}\tilde h_i)\otimes\LL_{\rho^{de/c}}\otimes\UU_n\otimes\UU_m\hookrightarrow
\R^1\Phi({\mathcal H})_{(\alpha_i,\infty)}.
$$

Taking the direct sum over all $i=0,\ldots,bd+ae-1$ we have

$$
\bigoplus_{i=0}^{bd+ae-1}\LL_\psi(h_i)\otimes\LL_{\rho^{de/c}}\otimes\UU_n\otimes\UU_m\hookrightarrow \bigoplus_{i=0}^{bd+ae-1}\R^1\Phi(\HH)_{(\alpha_i,\infty)}\hookrightarrow
$$
$$
\hookrightarrow\HHH^1(\PP^1_\infty,\R\Phi(\HH))
$$
and, tensoring with $\LL_{\chi^{e'}\xi^{d'}}$,
$$
\bigoplus_{i=0}^{bd+ae-1}\LL_\psi(h_i)\otimes\LL_{\chi^{e'}\xi^{d'}}\otimes\LL_{\rho^{de/c}}\otimes\UU_n\otimes\UU_m\hookrightarrow
[bd'+ae']^\ast LC_{(\infty,\infty)}^{\infty}(\FF,\GGG)
$$
 
But from \cite[6.1]{katz1988gauss} we know that $LC_{(\infty,\infty)}^{\infty}(\FF,\GGG)$ (and hence also $[bd'+ae']^\ast LC_{(\infty,\infty)}^{\infty}(\FF,\GGG)$) has dimension $mn(bd+ae)$, so we conclude that these inclusions must be isomorphisms. In particular, this shows that $\R^1\Phi(\HH)$ is supported exactly on the set $\{(\alpha_i,\infty)|i=0,\ldots,bd+ae-1\}$ and that $\R^1\Phi(\HH)_{(\alpha_i,\infty)}\cong\LL_\psi(h_i)\otimes\LL_{\rho^{de/c}}\otimes\UU_n\otimes\UU_m$ as a representation of $I_\infty$ for every $i=0,\ldots,bd+ae-1$.
\end{proof}

Let ${\mu}\subset k$ be the group of $(bd'+ae')$-th roots of unity. It acts on $\eta$ by multiplication, and the sheaf $[bd'+ae']^\ast(\FF\ast\GGG)$ is equivariant for this action. This action can be lifted to an action on $\GG_{m,\eta}$ by defining $\zeta\cdot z=\zeta^{\alpha d'-\beta e'}z$ for $\zeta\in\mu$. Note that the sheaf
$$
\LL_{\psi}(f(z^bt^{e'})+g(z^{-a}t^{d'}))\otimes\KK_{\chi,n}(z^bt^{e'})\otimes\KK_{\xi,m}(z^{-a}t^{d'})
$$
is invariant under this action, since $\zeta\cdot z^bt^{e'}=\zeta^{\alpha bd'-\beta be'+e'} z^bt^{e'}=\zeta^{\alpha bd'+\alpha ae'} z^bt^{e'}=z^bt^{e'}$ (and similarly for $z^{-a}t^{d'}$).

The actions of $\mu$ on $\GG_{m,\eta}$ and $\eta$ are compatible, so the action of $\mu$ on $[bd'+ae']^\ast(\FF\ast\GGG)\cong \R^1\pi_{2!}\left(\LL_{\psi}(f(z^bt^{e'})+g(z^{-a}t^{d'}))\otimes\KK_{\chi,n}(z^bt^{e'})\otimes\KK_{\xi,m}(z^{-a}t^{d'})\right)$ is induced by its action on the pair $$(\GG_{m,\eta},
\LL_{\psi}(f(z^bt^{e'})+g(z^{-a}t^{d'}))\otimes\KK_{\chi,n}(z^bt^{e'})\otimes\KK_{\xi,m}(z^{-a}t^{d'}))
$$
From Lemma \ref{l1} we know that
$$
\R^1\pi_{2!}\left(\LL_{\psi}(f(z^bt^{e'})+g(z^{-a}t^{d'}))\otimes\KK_{\chi,n}(z^bt^{e'})\otimes\KK_{\xi,m}(z^{-a}t^{d'})\right)\cong
$$
$$
\cong\LL_{\chi^{e'}\xi^{d'}}\otimes\HHH^1(\PP^1_{\infty},\R\Phi(\HH))\cong\bigoplus_{i=0}^{bd+ae-1}\LL_{\chi^{e'}\xi^{d'}}\otimes\R^1\Phi(\HH)_{(\alpha_i,\infty)}
$$
and the action of $\zeta\in\mu$ takes $(\alpha_i,\infty)$ to $(\zeta^{\alpha d'-\beta e'}\alpha_i,\infty)=(\alpha_{i+c(\alpha d'-\beta e')},\infty)=(\alpha_{i+\alpha d-\beta e},\infty)$ (where we define $\alpha_j=\alpha_i$ if $j\equiv i$ mod $bd+ae$). So the action of $\zeta$ permutes the summands of $\LL_{\chi^{e'}\xi^{d'}}\otimes\HHH^1(\PP^1_{\infty},\R\Phi(\HH))$ by taking $\LL_{\chi^{e'}\xi^{d'}}\otimes\R^1\Phi(\HH)_{(\alpha_i,\infty)}$ to $\LL_{\chi^{e'}\xi^{d'}}\otimes\R^1\Phi(\HH)_{(\alpha_{i+\alpha d-\beta e},\infty)}$.

In particular, $\zeta^l$ fixes $\LL_{\chi^{e'}\xi^{d'}}\otimes\R^1\Phi(\HH)_{(\alpha_i,\infty)}$ if and only if $bd+ae$ divides $l(\alpha d-\beta e)$, that is, if and only if $bd'+ae'$ divides $l(\alpha d'-\beta e')$. But $bd'+ae'$ and $\alpha d'-\beta e'$ are relatively prime, since $\beta(bd'+ae')+a(\alpha d'-\beta e')=d'$ and $\alpha(bd'+ae')-b(\alpha d'-\beta e')=e'$ and $d',e'$ are relatively prime. Therefore $bd'+ae'$ must divide $l$, so $\mu$ acts freely on the set of summands of $\LL_{\chi^{e'}\xi^{d'}}\otimes\HHH^1(\PP^1_{\infty},\R\Phi(\HH))$. 

We deduce that, as a representation of $I_\infty$, $LC_{(\infty,\infty)}^\infty(\FF,\GGG)$ is induced from the direct sum of representatives of the orbits of the action of $\mu$ on the set of summands, that is,
$$
LC_{(\infty,\infty)}^{\infty}(\FF,\GGG) \cong[bd'+ae']_\ast\left(\bigoplus_{i=0}^{c-1}\LL_\psi(h_i)\otimes\LL_{\rho^{de/c}}\otimes\LL_{\chi^{e'}\xi^{d'}}\otimes\UU_n\otimes\UU_m\right)
$$
This proves Theorem \ref{t1}, since $de/c$ is even if and only if one of $d$, $e$ is even, that is, if and only if $de$ is even.

\section{The local convolution $LC_{(0,\infty)}^{\infty}$}

In this section we will deduce a similar explicit formula for the local convolution operator $LC_{(0,\infty)}^\infty$. The computations are very similar to those of the previous section, so we will not describe them in detail and only highlight the differences.

By \cite[Proposition 13]{rojas2013local}, $LC_{(0,\infty)}^\infty(\iota^\ast\FF,\GGG)=0$ if $d/a$ (the slope of $\FF$) is less than or equal to $e/b$ (the slope of $\GGG$), so we will assume that $bd>ae$. Again we will start by studying the restriction to the subgroup of index $bd'-ae'$ of $I_\infty$:

\begin{proposition} Under the previos notation, suppose that $bd-ae$ is prime to $p$. Then
 $$
[bd'-ae']^\ast LC_{(0,\infty)}^{\infty}(\iota^\ast\FF,\GGG)\cong\bigoplus_{i=0}^{bd-ae-1}(\LL_{\psi(h_i)}\otimes\LL_{\rho^{de/c}}\otimes\LL_{\chi^{e'}\xi^{d'}}\otimes \UU_n\otimes \UU_m)
$$
\end{proposition}

As a sheaf on $\Gm$, we have (using the same notation as in the previous section)
$$(\iota^\ast\FF)\ast\GGG=\R^1\mu_!([a]_\ast\iota^\ast(\LL_{\psi(f)}\otimes\KK_{\chi,n})\boxtimes[b]_\ast(\LL_{\psi(g)}\otimes\KK_{\xi,m}))=
$$
$$
=\R^1\sigma_!(\iota^\ast(\LL_{\psi(f)}\otimes\KK_{\chi,n})\boxtimes(\LL_{\psi(g)}\otimes\KK_{\xi,m}))
$$

Let $\alpha,\beta\in{\mathbb Z}$ be integers such that $\alpha a+\beta b=1$. Using the automorphism $\phi:\Gm\times\Gm\to\Gm\times\Gm$ given by $\phi(w,t)=(w^bt^\alpha,w^{-a}t^\beta)$ we get

$$(\iota^\ast\FF)\ast\GGG\cong\R^1\pi_{2!}\phi^\ast(\LL_{\psi(f)}(x^{-1})\otimes\KK_{\chi,n}(x^{-1})\otimes\LL_{\psi(g)}(t)\otimes\KK_{\xi,m}(t))=
$$
$$
=\R^1\pi_{2!}(\LL_{\psi(f)}(w^{-b}t^{-\alpha})\otimes\KK_{\chi,n}(w^{-b}t^{-\alpha})\otimes\LL_{\psi(g)}(w^{-a}t^\beta)\otimes\KK_{\xi,m}(w^{-a}t^\beta))
$$

By proper base change, we have
$$
[bd'-ae']^\ast((\iota^\ast\FF)\ast\GGG)\cong
$$
$$
\cong\R^1\pi_{2!}(\LL_{\psi(f)}(w^{-b}t^{-\alpha(bd'-ae')})\otimes\KK_{\chi,n}(w^{-b}t^{-\alpha(bd'-ae')})\otimes
$$
$$
\otimes\LL_{\psi(g)}(w^{-a}t^{\beta(bd'-ae')})\otimes\KK_{\xi,m}(w^{-a}t^{\beta(bd'-ae')}))=
$$
$$
=\R^1\pi_{2!}(\LL_{\psi(f)}((wt^{\alpha d'+\beta e'})^{-b}t^{e'})\otimes\KK_{\chi,n}((wt^{\alpha d'+\beta e'})^{-b}t^{e'})\otimes
$$
$$
\otimes\LL_{\psi(g)}((wt^{\alpha d'+\beta e'})^{-a}t^{d'})\otimes\KK_{\xi,m}((wt^{\alpha d'+\beta e'})^{-a}t^{d'})
$$

We now specialize at $S$, the henselization of $\PP^1$ at $\infty$, and use the $\eta$-automorphism $\varphi:\GG_{m,\eta}\to\GG_{m,\eta}$ given by $w\mapsto wt^{\alpha d'+\beta e'}$ to obtain (as sheaves on $S$):
$$
[bd'-ae']^\ast((\iota^\ast\FF)\ast\GGG)\cong
$$
$$
\cong\R^1\pi_{2!}\left(\LL_{\psi(f)}(z^{-b}t^{e'})\otimes\KK_{\chi,n}(z^{-b}t^{e'})
\otimes\LL_{\psi(g)}(z^{-a}t^{d'})\otimes\KK_{\xi,m}(z^{-a}t^{d'})\right)=
$$
$$
=\R^1\pi_{2!}\left(\LL_{\psi}(f(z^{-b}t^{e'})+g(z^{-a}t^{d'}))\otimes\KK_{\chi,n}(z^{-b}t^{e'})\otimes\KK_{\xi,m}(z^{-a}t^{d'})\right)\cong
$$
$$
\cong\LL_{\chi^{e'}\xi^{d'}}\otimes\R\bar\pi_{2\ast}\HH=\LL_{\chi^{e'}\xi^{d'}}\otimes\R\Gamma(\PP^1_\infty,\R\Phi(\HH))
$$
where $\HH$ is the sheaf $\LL_{\psi}(f(z^{-b}t^{e'})+g(z^{-a}t^{d'}))\otimes\LL_{\chi^b\xi^{-a}}(z)\otimes\UU_n(z^bt^{-e'})\otimes\UU_m(z^{-a}t^{d'})$ on $\GG_{m,\eta}$, extended by zero to $\PP^1_S$. 

The proof of the proposition is completed by the following lemma, whose proof is identical to that of Lemma \ref{l1} (using that $LC_{(0,\infty)}^\infty(\iota^\ast\FF,\GGG)$ has dimension $bd-ae$ by \cite[Proposition 13]{rojas2013local}).

\begin{lemma}
 $\R^i\Phi(\HH)=0$ for $i\neq 1$. The sheaf $\R^1\Phi(\HH)$ is supported on the set $\overline Z$ of $(bd-ae)$-th roots of $-bdf_d/aeg_e$. For every such root $\alpha_i$, let $z_i(t^{-1})\in k[t^{-1}]^h_{(t^{-1})}$ be the only root of $\frac{\partial}{\partial z}H(z,t)$ such that $z_i(t^{-1})\equiv\alpha_i$ mod $t^{-1}$, and $h_i(t)=t^{de/c}H(z_i(t^{-1}),t)\in t^{de/c} k[[t^{-1}]]$. Then $\R^1\Phi(\HH)_{(\alpha_i,\infty)}$ has dimension $mn$, and $I_\infty$ acts on it via the representation $\LL_{\psi}(h_i)\otimes\LL_{\rho^{de/c}}\otimes\UU_n\otimes\UU_m$.
\end{lemma}

Let $\mu$ be the group of $(bd'-ae')$-th roots of unity. It acts on $\eta$ by multiplication, and the sheaf $[bd'-ae']^\ast((\iota^\ast\FF)\ast\GGG)$ on $S$ is equivariant for this action. We lift this action to an action on $\GG_{m,\eta}$ by defining $\zeta\cdot z=\zeta^{\alpha d'+\beta e'}z$ for $\zeta\in\mu$. Then the sheaf
$$
\LL_{\psi}(f(z^{-b}t^{e'})+g(z^{-a}t^{d'}))\otimes\KK_{\chi,n}(z^{-b}t^{e'})\otimes\KK_{\xi,m}(z^{-a}t^{d'})
$$
is invariant under this action, and the action of $\mu$ on $[bd'-ae']^\ast((\iota^\ast\FF)\ast\GGG)\cong  \R^1\pi_{2!}\left(\LL_{\psi}(f(z^{-b}t^{e'})+g(z^{-a}t^{d'}))\otimes\KK_{\chi,n}(z^{-b}t^{e'})\otimes\KK_{\xi,m}(z^{-a}t^{d'})\right)$ is induced by its action on the pair $$(\GG_{m,\eta},
\LL_{\psi}(f(z^{-b}t^{e'})+g(z^{-a}t^{d'}))\otimes\KK_{\chi,n}(z^{-b}t^{e'})\otimes\KK_{\xi,m}(z^{-a}t^{d'}))
$$
Since
$$
\R^1\pi_{2!}\left(\LL_{\psi}(f(z^{-b}t^{e'})+g(z^{-a}t^{d'}))\otimes\KK_{\chi,n}(z^{-b}t^{e'})\otimes\KK_{\xi,m}(z^{-a}t^{d'})\right)\cong
$$
$$
\cong\LL_{\chi^{e'}\xi^{d'}}\otimes
\HHH^1(\PP^1_{\infty},\R\Phi(\HH))\cong\bigoplus_{i=0}^{bd-ae-1}\LL_{\chi^{e'}\xi^{d'}}\otimes\R^1\Phi(\HH)_{(\alpha_i,\infty)}
$$
and the action of $\zeta\in\mu$ takes $(\alpha_i,\infty)$ to $(\zeta^{\alpha d'+\beta e'}\alpha_i,\infty)=(\alpha_{i+c(\alpha d'+\beta e')},\infty)=(\alpha_{i+\alpha d+\beta e},\infty)$, the action of $\zeta$ permutes the summands of $\LL_{\chi^{e'}\xi^{d'}}\otimes\HHH^1(\PP^1_{\infty},\R\Phi(\HH))$ by taking $\LL_{\chi^{e'}\xi^{d'}}\otimes\R^1\Phi(\HH)_{(\alpha_i,\infty)}$ to $\LL_{\chi^{e'}\xi^{d'}}\otimes\R^1\Phi(\HH)_{(\alpha_{i+\alpha d+\beta e},\infty)}$. As in the previous section, since $bd'-ae'$ and $\alpha d' +\beta e'$ are relatively prime, we conclude that $\mu$ acts freely on the set of summands of $\LL_{\chi^{e'}\xi^{d'}}\otimes\HHH^1(\PP^1_{\infty},\R\Phi(\HH))$. 

Therefore
$$
LC_{(0,\infty)}^{\infty}(\iota^\ast\FF,\GGG)\cong[bd'-ae']_\ast\left(\bigoplus_{i=0}^{c-1}\LL_\psi(h_i)\otimes\LL_{\rho^{de/c}}\otimes\LL_{\chi^{e'}\xi^{d'}}\otimes\UU_n\otimes\UU_m\right).
$$
This proves Theorem \ref{t2}, since $de/c$ is even if and only if $de$ is.

\section{Local Fourier transform}

The Fourier transform is related to the multiplicative convolution via the formula \cite[Proposition 8.1.12]{katz1990esa}
$$
K\ast\LL_\psi\cong j_!(\mathrm{FT}_\psi(\iota^\ast j^\ast K))
$$
for every object $K\in\Dbc(\Gm,\Ql)$, where $j:\Gm\to\AAA^1$ is the inclusion. Using this, we can recover the explicit formulas for the local Fourier transform given by Fu in \cite{fu2010calculation} from our formulas for the local convolution, removing the hypothesis that $p$ is large enough:

\begin{corollary}
 Let $\FF=[a]_\ast(\LL_\psi(f(1/t))\otimes\KK_{\chi,n})$ be a representation of $I_0$, where $f\in k[t]$ has degree $d$. Suppose that $a$, $d$ and $a+d$ are prime to $p$. Let $z(t^{-1})\in k[[t^{-1}]]$ be a root of $z^{a+1}t^{-(d-1)}f'(tz)-a\in k[[t^{-1}]][z,z^{-1}]$, and let $g(t)=f(t\cdot z(t^{-1}))+t^d/z(t^{-1})^a\in t^d\cdot k[[t^{-1}]]$. Then
$$
\mathrm{FT}_{(0,\infty)}\FF\cong [d+a]_\ast(\LL_\psi(g)\otimes\KK_{\bar\chi,n}\otimes\LL_{\rho^d})
$$
\end{corollary}

\begin{proof} This is just a particular case of Theorem \ref{t1}, applied to $\iota^\ast\FF\cong [a]_\ast(\LL_\psi(f)\otimes\KK_{\bar\chi,n})$ and $\LL_\psi$. Note that $d$ and $e=1$ are relatively prime in this case, so $c=1$.
\end{proof}

Let us check that this result is equivalent to \cite[Theorem 0.2]{fu2010calculation}: if $z(t^{-1})$ is a root of $z^{a+1}t^{-(d-1)}f'(tz)-a$, then $t=1/t'z(t'^{-1})$ satisfies
$$
\frac{d}{dt}f(1/t)+at^{a-1}t'^{a+d}=0
$$
and $f(1/t)+t^at'^{a+d}$ coincides with the $g$ in the corollary. Note that \cite[Theorem 0.2]{fu2010calculation} has the additional hypothesis that $p>d$, which is not needed here.

Similarly, we can deduce a formula for the local Fourier transform $\mathrm{FT}_{(\infty,\infty)}$:

\begin{corollary}
 Let $\FF=[a]_\ast(\LL_\psi(f(t))\otimes\KK_{\chi,n})$ be a representation of $I_\infty$, where $f\in k[t]$ has degree $d$ and $d>a$. Suppose that $a$, $d$ and $d-a$ are prime to $p$. Let $z(t^{-1})\in k[[t^{-1}]]$ be a root of $z^{a-1}t^{-(d-1)}f'(t/z)+a\in k[[t^{-1}]][z,z^{-1}]$, and let $g(t)=f(t/z(t^{-1}))+t^d/z(t^{-1})^a\in t^d\cdot k[[t^{-1}]]$. Then
$$
\mathrm{FT}_{(\infty,\infty)}\FF\cong [d-a]_\ast(\LL_\psi(g)\otimes\KK_{\chi,n}\otimes\LL_{\rho}^d)
$$
\end{corollary}

\begin{proof} This is just a particular case of Theorem \ref{t2}, applied to $\iota^\ast\FF$ and $\LL_\psi$.
\end{proof}

Again, this result is equivalent to \cite[Theorem 0.3]{fu2010calculation}, with the advantadge that the hypothesis $p>d$ is not needed.

By using the recursive formulas given in \cite{rojas2013local}, from the formulas for the local convolutions $LC_{(\infty,\infty)}^\infty$ and $LC_{(0,\infty)}^\infty$ and the ones for the local Fourier transform, one can derive explicit formulas for the remaining local convolution functors $LC_{(a,b)}^c$.

\bibliographystyle{amsalpha}
\bibliography{bibliography}

\end{document}